\documentclass[12pt]{amsart}

\usepackage[margin=1in]{geometry}
\usepackage{amsmath,amssymb,amsthm,mathtools}
\usepackage{enumitem}
\usepackage{float}
\usepackage{microtype}
\usepackage[hidelinks]{hyperref}
\usepackage{tikz-cd}
\usetikzlibrary{arrows.meta,positioning,calc,fit,shadows.blur}
\hypersetup{pdftitle={Braided Surface and Small Exotic 4-Manifolds},
  pdfauthor={Anar Akhmedov},
  pdfsubject={Braided quotient blocks and small exotic four-manifolds},
  pdfkeywords={branched covers, braided surfaces, exotic four-manifolds, symplectic sums}}

\definecolor{branchred}{RGB}{235,45,82}
\definecolor{branchblue}{RGB}{25,105,235}
\definecolor{gluegreen}{RGB}{20,175,95}
\definecolor{turquoise}{RGB}{0,185,205}
\definecolor{blockpurple}{RGB}{145,70,225}
\definecolor{warmorange}{RGB}{245,135,25}

\newtheorem{theorem}{Theorem}[section]
\newtheorem{proposition}[theorem]{Proposition}
\newtheorem{lemma}[theorem]{Lemma}
\newtheorem{corollary}[theorem]{Corollary}
\newtheorem{question}[theorem]{Question}
\theoremstyle{definition}

\newtheorem{definition}[theorem]{Definition}
\theoremstyle{remark}
\newtheorem{remark}[theorem]{Remark}

\newcommand{\CP}{\mathbb{CP}}
\newcommand{\CPb}{\overline{\mathbb{CP}}}
\newcommand{\Z}{\mathbb Z}

\newcommand{\Mod}{\operatorname{Mod}}
\newcommand{\Diff}{\operatorname{Diff}}

\newcommand{\id}{\operatorname{id}}

\title[Braided Surface and Small Exotic 4-Manifolds]{Braided Surface and Small Exotic $4$-Manifolds}
\author{Anar Akhmedov}
\address{School of Mathematics, University of Minnesota, Minneapolis, MN 55455, USA}
\address{Department of Mathematics, Harvard University, Cambridge, MA 02138, USA}
\email{akhmedov@umn.edu}
\date{}

\begin{document}

\begin{abstract}
Let $K$ be the trefoil or figure-eight knot and let $M_K$ denote zero surgery
on $K$.  We prove that the fiberwise elliptic involution gives a two-fold
branched cover
\[
 M_K\times S^1\longrightarrow T^2\times S^2
\]
whose branch locus is a section together with a connected braided
three-fold multisection.  The corresponding braid words are computed explicitly for
both knots and make visible the relation with Luttinger surgery along
product Lagrangian tori.
More generally, for a genus-$g$ fibered knot with hyperelliptic monodromy,
the same construction gives a degree-$(2g+2)$ braided multisection; this includes every
genus-two fibered knot.  We also state the corresponding surface-bundle
version over $T^2$.
We then track the induced peripheral data and gluing diffeomorphisms in the
constructions of $Y_K$ and its twisted double $X_K$, as studied by the author in
\cite[Section~4]{Akh07}, and through the constructions of the small exotic
manifolds homeomorphic to $\CP^2\#k\CPb^2$, $k=3,5$
\cite{AkhmedovSmall,ABP,AP2008}.  In the symmetric construction of a
$k=5$ example \cite[Section~3.3]{ABP}, each summand admits a double cover
branched along a surface, although the marked complements used in the
twisted double need not be invariant.  Finally, we prove boundary-action,
relative-homology, and local-smoothing obstructions to assembling these
specified local involutions.  We do not construct a global branched-cover
model of $X_5$, $X_5^{\rm sym}$, or $X_3$.
\end{abstract}

\maketitle

\section{Introduction}

The constructions in \cite{Akh07,AkhmedovSmall,ABP,AP2008} considered here
use copies of $M_K\times S^1$ as principal building blocks, together with
the additional blocks recalled below.  Here $M_K$ denotes zero surgery on
the fibered knot $K$.  We study the quotient geometry for $K$ equal to the
right-handed trefoil or the figure-eight knot.

A global two-fold branched-cover description would replace a closed
four-manifold assembled by symplectic sums with explicit branch-surface data
in a simpler base.  The difficulty is not the quotient of an individual
block, but the equivariant compatibility of the boundary actions and marked
surfaces used in the sums.  This paper computes the local quotients, carries
their markings through the constructions, and proves obstructions to gluing
the natural local actions by the original sum data.

The elliptic involution on a torus commutes with the chosen linear
representative of the monodromy of $M_K$ and gives
\[
 M_K\times S^1\longrightarrow T^2\times S^2.
\]
The branch surface is a degree-four braided multisection.  In both knot
cases it splits as a section $B_1$ and a connected three-fold multisection $B_3$.
The origin in the two-torsion is fixed and the three nonzero points form one
orbit.  With the marking and product convention fixed in Section~2, the
braid monodromies may be represented by
\[
 \beta_{\rm tr}=h_\beta h_\alpha,
 \qquad \beta_{\rm fig}=h_\alpha h_\beta^{-1}.
\]
Here $h_\alpha$ and $h_\beta$ are the marked positive half-twists defined
there, and products act from right to left.
These formulas also display the relation with Luttinger surgery along
product Lagrangian tori:
composition of the torus monodromy with a Dehn twist is represented in the
quotient by
composition of the four-point spherical braid with the corresponding
half-twist; compare \cite{ADK}.

We compute the local quotients and their marked boundary data, then apply
them to $Y_K$, $X_K$, $X_5$, and $X_3$.  These closed manifolds are those of
the cited papers; the new point is the obstruction to assembling their
natural local involutions.

The closed constructions retain their original provenance.  The fiber sum
$Y_K$, which identifies a fiber torus in one copy with a section torus in
the other and is called the fiber-section sum below, and the knot-indexed
twisted double $X_K$, with their marked gluings, are due to \cite{Akh07}.
The particular $X_5$ construction
recalled here is the sum in \cite{AkhmedovSmall}, which gives a manifold
homeomorphic to $\CP^2\#5\CPb^2$ by summing $Y_K$ with the Matsumoto block.
The related symmetric construction in \cite[Section~3.3]{ABP} expresses a
$k=5$ example, denoted here by $X_5^{\rm sym}$, as a genus-two twisted double
of $(M_K\times S^1)\#2\CPb^2$.  The particular $X_3$ model used below is the
construction of \cite{AP2008}, and the families denoted here by $X_2(m)$ and
$X_4(m)$ are the corresponding families from \cite{APodd}.  The
fundamental-group and minimality arguments, together with gauge-theoretic
arguments where applicable, are recalled from the indicated cited
constructions.

Our notation separates two uses of a subscript.  The symbol $X_K$ denotes
the knot-indexed twisted double defined in Section~\ref{subsec:XK-double}.
By contrast, the numerical subscript in $X_5$ or $X_3$ records the number
$k$ in the homeomorphism type $\CP^2\#k\CPb^2$; these manifolds are not
specializations of $X_K$.  Similarly, $X_2(m)$ and $X_4(m)$ denote the
one-parameter families homeomorphic to $\CP^2\#2\CPb^2$ and
$\CP^2\#4\CPb^2$, respectively, with $m$ the surgery parameter.

For the closed constructions in the remainder of this Introduction, $K$ is
the right-handed trefoil.  We record the local quotient pairs, their
$\mathbb Z/2$ monodromy homomorphisms, and the gluing diffeomorphisms in the
covering manifolds for $Y_K$, its twisted double $X_K$, $X_5$, and $X_3$.
These local involutions do not, under the original gluing maps, define a
single orientation-preserving involution with two-dimensional fixed set on
the closed manifold.

The involutions induced on the boundaries of a fiber neighborhood and a
fixed-section neighborhood are not conjugate.  A
branched cobordism carrying the prescribed boundary actions and covering
characters must have a double cover with nontrivial relative homology.
Moreover, the natural elliptic action does not preserve the connected marked
relative smoothing of the fiber and section used in either the
$X_5^{\rm sym}$ construction or
the $X_3$ construction.  Thus the results give necessary compatibility
conditions for extending the local involutions and isolate the data that the
original gluing construction does not preserve.

\begin{theorem}\label{thm:main-local}
Let $K$ be the trefoil or figure-eight knot.  The fiberwise elliptic
involution defines a double branched cover
\[
 M_K\times S^1\longrightarrow T^2\times S^2.
\]
Its branch surface is the disjoint union $B_1\sqcup B_3$ of a section and a
connected three-fold multisection.  Both components are tori, and
\[
 [B_1]+[B_3]=4[T^2\times\{\mathrm{pt}\}]
 \quad\text{in }H_2(T^2\times S^2;\mathbb Z),
 \qquad B_1^2=B_3^2=0.
\]
\end{theorem}

The next theorem concerns the specified local involutions and marked
gluings; it does not rule out a different involution on a closed manifold.

\begin{theorem}\label{thm:main-obstruction}
The involutions induced on the boundaries of a fiber neighborhood and a
fixed-section neighborhood are not equivariantly diffeomorphic.  If a
fiber--section branch cobordism realizes these boundary covers, the two
boundary inclusions into its double cover cannot both be isomorphisms on
$H_1(-;\mathbb Z)$.  In the $X_5^{\rm sym}$ and $X_3$ constructions, the
natural local involutions do not preserve the chosen connected relative
smoothings with their prescribed boundary collars.  For the natural
$X_3$ elliptic involution, no connected invariant relative smoothing has
these collars.
\end{theorem}

Consequently, the original local branched covers cannot simply be glued
through the published sum data.  This conclusion does not exclude a
different involution or a non-product interpolating cobordism.

The torus-bundle quotient, braid words, and the stated boundary-action and
homology obstructions apply to both knots.  The detailed $X_K$, $X_5$, and
$X_3$ fundamental-group calculations use the right-handed trefoil, as in
\cite{AkhmedovSmall,ABP,AP2008}; no figure-eight analogue of those closed
examples is asserted here.

Throughout, $\approx$ denotes orientation-preserving homeomorphism.  All
fiber-sum gluings use the standard meridian-reversing convention.

Sections~2--3 establish the torus-bundle quotient and the marked descriptions
of $Y_K$ and $X_K$.  Sections~4--5 develop the equivariant-sum criterion and
the boundary obstruction.  Section~6 treats the Matsumoto block.  Section~7
treats $X_5$ and $X_5^{\rm sym}$, Section~8 treats the $X_3$
construction, and the final section gives numerical restrictions on possible
quotient manifolds.

\section{\texorpdfstring{$M_K\times S^1$ as a double branched cover}{The torus-bundle block as a double branched cover}}
\label{sec:MK-cover}

We first record the higher-genus form of the construction.  Let $K$ be a
fibered knot of genus $g\geq1$, let $\widehat\Sigma_g$ be the closed fiber after
zero surgery, and let
\[
 \varphi_K:\widehat\Sigma_g\longrightarrow\widehat\Sigma_g
\]
be its capped monodromy.  Thus
\[
 M_K=\widehat\Sigma_g\times_{\varphi_K}S^1.
\]
Call the monodromy \emph{hyperelliptic} if its mapping class admits a
diffeomorphism representative $f$ commuting with a fixed hyperelliptic
involution $\iota_g$ of $\widehat\Sigma_g$.  We form the mapping torus using
such a representative; replacing the monodromy by an isotopic representative
does not change its diffeomorphism type.

\begin{theorem}\label{thm:general-hyperelliptic-block}
Suppose that the capped monodromy of a genus-$g$ fibered knot $K$ is
hyperelliptic in the preceding sense, and choose $f$ with
$f\iota_g=\iota_gf$.  Then $\iota_g\times\id_{S^1}$ defines an involution
on $(\widehat\Sigma_g\times_fS^1)\times S^1$, identified with
$M_K\times S^1$, and its quotient map is a two-fold branched cover
\[
 M_K\times S^1\longrightarrow T^2\times S^2.
\]
The branch locus is a braided multisection of degree $2g+2$.  Its connected
components are tori indexed by the orbits of the induced permutation of the
$2g+2$ Weierstrass points.  If $A=[T^2\times\{\mathrm{pt}\}]$, then
\[
 [B_K]=(2g+2)A,\qquad B_K^2=0.
\]
In particular, the conclusion holds for every genus-one and every genus-two
fibered knot.
\end{theorem}

\begin{proof}
For the chosen commuting representative $f$, the involution $\iota_g$
descends to the mapping torus and acts trivially on the additional circle.
The quotient of
$\widehat\Sigma_g$ by $\iota_g$ is $S^2$, branched at its $2g+2$
Weierstrass points.  Hence the quotient of $M_K\times S^1$ is
\[
 (S^2\times_{\bar f}S^1)\times S^1.
\]
The induced map $\bar f$ preserves the marked branch set and is
orientation preserving.  After forgetting the marked points it is isotopic
to the identity because $\Diff^+(S^2)$ is path connected
\cite{Smale}.  Choosing such an isotopy gives a noncanonical
fiber-preserving identification of the ambient quotient with
$T^2\times S^2$, while the motion of the marked points is retained as a
spherical braid.

Near a fixed point, the involution has the local form
$(z,u)\mapsto(-z,u)$, where $z$ is the oriented normal coordinate to the
fixed surface and $u$ is tangent to it.  Its quotient is smooth, and the image of the fixed
set is a smoothly embedded branch surface.

An orbit of length $r$ traces a connected $r$-fold cover of the first base
circle.  After taking the additional circle it gives a torus mapping
unbranched to $T^2$ with degree $r$.  Thus $B_K$ is an embedded oriented
union of tori, $e(B_K)=0$, and $[B_K]\cdot F=2g+2$.  Under the chosen
product identification write $[B_K]=(2g+2)A+dF$, where
$F=[\{\mathrm{pt}\}\times S^2]$.  With self-intersection computed in the
quotient orientation, note that $M_K\times S^1$ is a product with a closed
oriented three-manifold and therefore has signature zero.  Since
$\sigma(T^2\times S^2)=0$, the double-cover signature formula gives
\[
 0=2\cdot0-\frac12B_K^2.
\]
Here $B_K^2$ denotes the total normal Euler number of the possibly
disconnected branch surface.  Hence $B_K^2=0$.  As
$((2g+2)A+dF)^2=2(2g+2)d$, it follows that $d=0$.

For $g=1$, every mapping class has a linear representative commuting with
the elliptic involution induced by $-I$.  For $g=2$, the hyperelliptic
mapping class group is the full mapping class group.  Equivalently, every
mapping class is represented by a diffeomorphism commuting with a fixed
hyperelliptic involution; this is the genus-two realization statement in
the Birman--Hilden theorem \cite{BirmanHilden}.  This proves the last
assertion.  For $g\geq3$, hyperellipticity is not automatic; the hypothesis
is a sufficient condition for this fiberwise construction.
\end{proof}

The same argument has a bundle form.

\begin{corollary}\label{cor:surface-bundle-hyperelliptic}
Let $E\to T^2$ be an oriented $\Sigma_g$-bundle whose structure group
reduces to
\[
 C_{\Diff^+(\Sigma_g)}(\iota_g)
 =\{h\in\Diff^+(\Sigma_g):h\iota_g=\iota_gh\}
\]
for a fixed hyperelliptic involution $\iota_g$.  Then the fiberwise maps
$\iota_g$ define an involution on $E$.  Its quotient is an oriented
$S^2$-bundle over $T^2$, and $E$ is its double cover branched along a
braided multisection of total degree $2g+2$.  The branch components
correspond to the orbits of the induced monodromy action on the Weierstrass
points.  For $E=(\widehat\Sigma_g\times_fS^1)\times S^1$, with $f$ as in
Theorem~\ref{thm:general-hyperelliptic-block}, the quotient sphere bundle is
the product $T^2\times S^2$.
\end{corollary}

\begin{proof}
The reduction of structure group makes the local fiberwise involutions agree
under all transition maps.  Taking the quotient fiber by fiber gives the
asserted sphere bundle and branch multisection.  The quotient need not be the
trivial $S^2$-bundle in general.  In the product-mapping-torus case, one
clutching map is the identity and the other is isotopic to the identity in
$\Diff^+(S^2)$, so the quotient bundle is trivial.
\end{proof}

For genus one, write
\[
 A_K\in SL(2,\Z)
\]
for the monodromy of the closed torus fiber obtained after zero surgery. Then
\[
 M_K=T^2\times_{A_K}S^1.
\]
Every linear $A_K$ commutes with $x\mapsto-x$, so
Theorem~\ref{thm:general-hyperelliptic-block} applies.

\begin{proposition}\label{prop:MKcover}
There is a double branched cover
\[
 \pi_K:M_K\times S^1\longrightarrow T^2\times S^2.
\]
Its branch locus $B_K$ is a degree-four braided multisection satisfying
\[
 [B_K]=4[T^2\times\{\mathrm{pt}\}],\qquad B_K^2=0.
\]
\end{proposition}

\begin{proof}
This is Theorem~\ref{thm:general-hyperelliptic-block} with $g=1$ and the
chosen product trivialization of the quotient sphere bundle.
\end{proof}

\begin{proposition}\label{prop:split}
For the standard trefoil and figure-eight monodromies,
\[
 B_K=B_1\sqcup B_3,
\]
where $B_1$ is a section and $B_3$ is a connected three-fold multisection. Both
components are tori.
\end{proposition}

\begin{proof}
The four fixed points of the elliptic involution are identified with
\[
 T^2[2]\cong(\Z/2)^2.
\]
The monodromy acts linearly on this set. For the standard trefoil and
figure-eight genus-one monodromies, its reduction modulo two has order three.
It fixes the origin and cyclically permutes the three nonzero vectors. The
orbit decomposition is therefore $1+3$.

The corresponding components of the branch locus are respectively a section
and a connected three-fold cover of $T^2$. Every finite unbranched cover of a
torus is again a torus.  Write $[B_1]=A+d_1F$ and
$[B_3]=3A+d_3F$.  Since the components are disjoint,
$B_1\cdot B_3=3d_1+d_3=0$, while $d_1+d_3=0$ follows from
$[B_K]=4A$.  Hence $d_1=d_3=0$; in particular
$[B_1]=A$, $[B_3]=3A$, and both components have square zero.
\end{proof}

\begin{remark}
The branch locus is a degree-four multisection, but in these examples it is
not connected. The $1+3$ decomposition is useful because the section component
provides a geometric marker in the quotient picture.
\end{remark}

For the chosen oriented basis of the closed torus fiber, take
\[
 A_{\mathrm{tr}}=\begin{pmatrix}1&1\\-1&0\end{pmatrix},
 \qquad
 A_{\mathrm{fig}}=\begin{pmatrix}2&1\\1&1\end{pmatrix},
\]
up to conjugacy and inversion according to orientation conventions.  Modulo two,
these become
\[
 \overline A_{\mathrm{tr}}=\begin{pmatrix}1&1\\1&0\end{pmatrix},
 \qquad
 \overline A_{\mathrm{fig}}=\begin{pmatrix}0&1\\1&1\end{pmatrix}.
\]
The integral matrices have different dynamical types.  The trefoil matrix
satisfies
\[
 A_{\mathrm{tr}}^3=-I,\qquad A_{\mathrm{tr}}^6=I,
\]
so it has order six in $SL(2,\mathbb Z)$ and order three in
$PSL(2,\mathbb Z)$.  By contrast, $\operatorname{tr}(A_{\mathrm{fig}})=3$,
so $A_{\mathrm{fig}}$ is hyperbolic and has infinite order.  Nevertheless,
their reductions modulo two both have order three:
\[
 \overline A_{\mathrm{tr}}^{\,3}
 =\overline A_{\mathrm{fig}}^{\,3}=I.
\]
More explicitly, their actions on the nonzero two-torsion points are
\[
 \overline A_{\mathrm{tr}}:
 e_1\longmapsto e_1+e_2\longmapsto e_2\longmapsto e_1,
\]
and
\[
 \overline A_{\mathrm{fig}}:
 e_1\longmapsto e_2\longmapsto e_1+e_2\longmapsto e_1.
\]
Thus each fixes $0\in T^2[2]$ and has one orbit of length three on the
nonzero two-torsion points, although the two matrices traverse that orbit in
opposite orders.  This gives the branch decomposition in
Proposition~\ref{prop:split}.

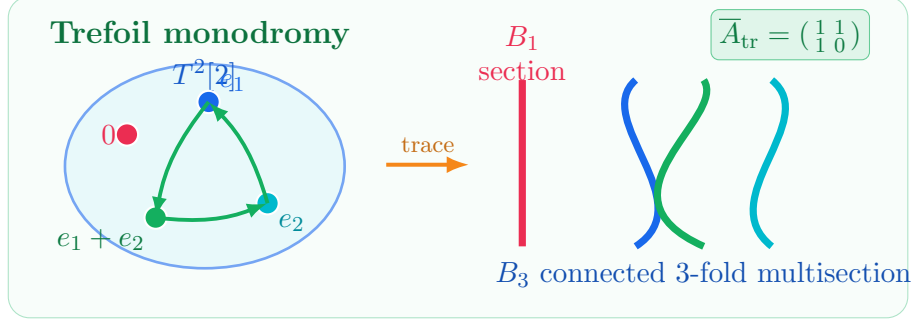
\begin{figure}[t]
\centering
\begin{tikzpicture}[>=Latex,font=\small,
  orbit/.style={line width=1.45pt,-{Latex[length=2.7mm]}},
  panel/.style={rounded corners=8pt,draw=gluegreen!35,fill=gluegreen!3,
    blur shadow={shadow opacity=.18,shadow blur steps=5}}]
 \node[panel,minimum width=11.9cm,minimum height=4.25cm] at (0,0) {};
 \node[anchor=west,font=\bfseries,text=gluegreen!65!black] at (-5.55,1.62)
   {Trefoil monodromy};
 \node[anchor=east,fill=gluegreen!13,draw=gluegreen!65,rounded corners=3pt,
   inner sep=3pt,text=gluegreen!70!black] at (5.45,1.62)
   {$\overline A_{\mathrm{tr}}=\left(\begin{smallmatrix}1&1\\1&0\end{smallmatrix}\right)$};
 \filldraw[fill=turquoise!9,draw=branchblue!60,line width=1pt]
   (-3.35,-.12) ellipse (1.85 and 1.35);
 \node[text=branchblue!85!black] at (-3.35,1.08) {$T^2[2]$};
 \coordinate (o) at (-4.38,.30);
 \coordinate (p) at (-3.30,.73);
 \coordinate (q) at (-4.00,-.80);
 \coordinate (r) at (-2.52,-.61);
 \filldraw[fill=branchred,draw=white,line width=.8pt] (o) circle (4.2pt);
 \node[left,text=branchred] at (o) {$0$};
 \filldraw[fill=branchblue,draw=white,line width=.8pt] (p) circle (4.2pt);
 \filldraw[fill=gluegreen,draw=white,line width=.8pt] (q) circle (4.2pt);
 \filldraw[fill=turquoise,draw=white,line width=.8pt] (r) circle (4.2pt);
 \node[above right,text=branchblue] at (p) {$e_1$};
 \node[below left,text=gluegreen!70!black] at (q) {$e_1+e_2$};
 \node[below right,text=turquoise!75!black] at (r) {$e_2$};
 \draw[orbit,gluegreen,bend right=14] (p) to (q);
 \draw[orbit,gluegreen,bend right=14] (q) to (r);
 \draw[orbit,gluegreen,bend right=14] (r) to (p);
 \draw[-{Latex[length=3mm]},warmorange,line width=1.5pt]
   (-.95,-.10)--node[above,font=\scriptsize,text=warmorange!80!black] {trace} (.15,-.10);
 \draw[branchred,line width=2.7pt] (.85,-1.18)--(.85,1.02);
 \node[text=branchred,align=center] at (.85,1.38) {$B_1$\\section};
 \draw[branchblue,line width=2.7pt] (2.35,-1.18) .. controls (3.25,-.60) and
 (1.75,.50) .. (2.35,1.02);
 \draw[gluegreen,line width=2.7pt] (3.25,-1.18) .. controls (1.72,-.50) and
 (3.58,.53) .. (3.25,1.02);
 \draw[turquoise,line width=2.7pt] (4.15,-1.18) .. controls (3.35,-.55) and
 (4.82,.48) .. (4.15,1.02);
 \node[align=center,text=branchblue!75!black] at (3.25,-1.55)
   {$B_3$ connected 3-fold multisection};
\end{tikzpicture}
\caption{Trefoil orbit picture.  The origin is fixed and traces the red
section $B_1$.  The green cyclic arrows join the three nonzero two-torsion
points into one orbit, whose trace is the connected 3-fold multisection $B_3$.}
\label{fig:trefoil-branch-orbit}
\end{figure}

\section{Local branched-cover models and marked gluing data}
\label{sec:piecewise-braids}

\begin{definition}\label{def:marked-block}
A \emph{marked local branched-cover model} consists of a two-fold quotient
$p:\widetilde V\to V$ of connected manifolds, its branch surface $B$, the
covering character $\chi_p:\pi_1(V\setminus B)\to\mathbb Z/2$, and the
peripheral markings used below.  A marked summing surface in $\widetilde V$
is a square-zero surface, not necessarily invariant, with an ordered
symplectic basis, a positive meridian, and basepoint paths in its complement.
When $V$ is a surface bundle and $B$ is a braided multisection, we call this
a \emph{braided local model}.
\end{definition}

For each construction we give the quotient pairs, braid monodromy, sheet
orbits, boundary markings, and gluing maps.  For surfaces created by
smoothing or blowup we also give their homology classes, self-intersections,
and meridians.  Since the gluing maps need not be equivariant, this local
information need not define a quotient of the assembled manifold.

\begin{remark}\label{rem:globalization-criterion}
For the specified local involutions and markings, the local actions extend
to the closed manifold only if every gluing map is isotopic, through marked
boundary maps, to an equivariant map and the
induced quotient maps identify
the branch pairs and their restricted covering characters.  A braid word
does not determine the normal-circle class.  For a construction involving
smoothings or blowups, the chosen smoothing and blowup data must also be
equivariant, with compatible local framings.
\end{remark}

Each copy of $M_K\times S^1$ contributes the quotient pair
\[
 (T^2\times S^2,B_K),\qquad B_K=B_1\sqcup B_3,
\]
together with the boundary component and peripheral marking used in the
relevant fiber sum.

Fix standard generators $s,t$ of the base torus, with $s$ the mapping-torus
direction.  The marked branch-point monodromy is
\[
 \rho_K:\pi_1(T^2)=\langle s,t\mid[s,t]=1\rangle
 \longrightarrow \Mod(S^2,P),
 \qquad \rho_K(s)=[\overline A_K],\quad \rho_K(t)=1,
\]
where $P$ is the four-point branch set and $\Mod(S^2,P)$ denotes mapping
classes preserving $P$ setwise.  To fix the marking explicitly,
write $T^2=\mathbb R^2/\mathbb Z^2$, let $a$ and $b$ be the horizontal and
vertical oriented curves through the origin, and identify
$T^2/\{\pm1\}$ with the pillowcase sphere.  Its marked points are the images
of
\[
 (0,0),\quad (1/2,0),\quad (0,1/2),\quad (1/2,1/2).
\]
Let $\alpha$ and $\beta$ be the images of $a$ and $b$; they are arcs with a
common endpoint at the image of the origin.  We define the positive
half-twists $h_\alpha,h_\beta$ by requiring that their lifts are the
right-handed Dehn twists $T_a,T_b$.  After choosing a trivialization, a
lift to the spherical braid group may be represented by
\[
 \beta_{\rm tr}=h_\beta h_\alpha,
 \qquad
 \beta_{\rm fig}=h_\alpha h_\beta^{-1},
\]
where products act from right to left.  We choose the isotopy trivializing
the ambient $S^2$ mapping torus so that the spherical braid traced by the
marked points is the
displayed word; changing this isotopy may multiply the braid by the central
full twist in the spherical braid group.  The
class in $\Mod(S^2,P)$ is intrinsic only up
to the stated marking conventions.  A representative in the spherical braid
group additionally depends on a choice of braid lift, while a lift of the
spherical mapping class to the covering torus is determined up to the
elliptic deck involution.  These are distinct ambiguities.  The braid data do
not determine the peripheral normal-circle class, which is recorded
separately.

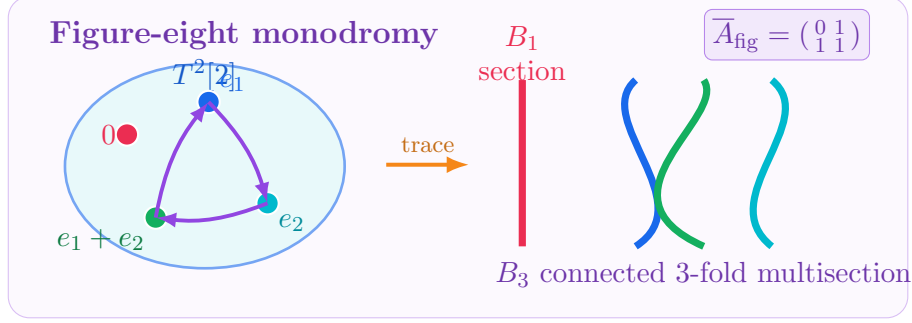
\begin{figure}[t]
\centering
\begin{tikzpicture}[>=Latex,font=\small,
  orbit/.style={line width=1.45pt,-{Latex[length=2.7mm]}},
  panel/.style={rounded corners=8pt,draw=blockpurple!35,fill=blockpurple!3,
    blur shadow={shadow opacity=.18,shadow blur steps=5}}]
 \node[panel,minimum width=11.9cm,minimum height=4.25cm] at (0,0) {};
 \node[anchor=west,font=\bfseries,text=blockpurple!75!black] at (-5.55,1.62)
   {Figure-eight monodromy};
 \node[anchor=east,fill=blockpurple!12,draw=blockpurple!65,rounded corners=3pt,
   inner sep=3pt,text=blockpurple!75!black] at (5.45,1.62)
   {$\overline A_{\mathrm{fig}}=\left(\begin{smallmatrix}0&1\\1&1\end{smallmatrix}\right)$};
 \filldraw[fill=turquoise!9,draw=branchblue!60,line width=1pt]
   (-3.35,-.12) ellipse (1.85 and 1.35);
 \node[text=branchblue!85!black] at (-3.35,1.08) {$T^2[2]$};
 \coordinate (o) at (-4.38,.30);
 \coordinate (p) at (-3.30,.73);
 \coordinate (q) at (-4.00,-.80);
 \coordinate (r) at (-2.52,-.61);
 \filldraw[fill=branchred,draw=white,line width=.8pt] (o) circle (4.2pt);
 \node[left,text=branchred] at (o) {$0$};
 \filldraw[fill=branchblue,draw=white,line width=.8pt] (p) circle (4.2pt);
 \filldraw[fill=gluegreen,draw=white,line width=.8pt] (q) circle (4.2pt);
 \filldraw[fill=turquoise,draw=white,line width=.8pt] (r) circle (4.2pt);
 \node[above right,text=branchblue] at (p) {$e_1$};
 \node[below left,text=gluegreen!70!black] at (q) {$e_1+e_2$};
 \node[below right,text=turquoise!75!black] at (r) {$e_2$};
 \draw[orbit,blockpurple,bend left=14] (p) to (r);
 \draw[orbit,blockpurple,bend left=14] (r) to (q);
 \draw[orbit,blockpurple,bend left=14] (q) to (p);
 \draw[-{Latex[length=3mm]},warmorange,line width=1.5pt]
   (-.95,-.10)--node[above,font=\scriptsize,text=warmorange!80!black] {trace} (.15,-.10);
 \draw[branchred,line width=2.7pt] (.85,-1.18)--(.85,1.02);
 \node[text=branchred,align=center] at (.85,1.38) {$B_1$\\section};
 \draw[branchblue,line width=2.7pt] (2.35,-1.18) .. controls (3.25,-.60) and
 (1.75,.50) .. (2.35,1.02);
 \draw[gluegreen,line width=2.7pt] (3.25,-1.18) .. controls (1.72,-.50) and
 (3.58,.53) .. (3.25,1.02);
 \draw[turquoise,line width=2.7pt] (4.15,-1.18) .. controls (3.35,-.55) and
 (4.82,.48) .. (4.15,1.02);
 \node[align=center,text=blockpurple!75!black] at (3.25,-1.55)
   {$B_3$ connected 3-fold multisection};
\end{tikzpicture}
\caption{Figure-eight orbit picture.  The fixed origin traces $B_1$.  The
purple cyclic arrows run through the three nonzero two-torsion points in the
order opposite to the trefoil arrows in Figure~\ref{fig:trefoil-branch-orbit}
and produce the connected 3-fold multisection $B_3$.}
\label{fig:figure-eight-branch-orbit}
\end{figure}

\subsection{Relation with Luttinger surgery}

For the linear fibered models used here, the standard Luttinger-surgery
calculation says that $\pm1$ surgery, measured relative to the
Lagrangian product framing, along the product of a fiber curve $c$ and a
base circle changes the monodromy around the base circle transverse to the
surgery torus by a factor $T_c^{\pm1}$.
The sign and the side of composition are determined by the orientation,
slope, and framing conventions; see \cite{ADK}.  The
chosen linear monodromy representatives still commute with $x\mapsto-x$, so
the resulting torus-bundle block again has the fiberwise quotient above.
For an invariant linear fiber curve $c$, its quotient is an arc $\bar c$
joining two branch points.  Under Birman-Hilden, the additional factor is
represented in the quotient by the signed half-twist
$h_{\bar c}^{\pm1}$.  In our conventions, with respect to the ordered
basis $(a,b)$,
\[
 T_a=\begin{pmatrix}1&1\\0&1\end{pmatrix},\qquad
 T_b=\begin{pmatrix}1&0\\-1&1\end{pmatrix}.
\]
Consequently,
\[
 A_{\rm tr}=T_bT_a,\qquad A_{\rm fig}=T_aT_b^{-1}
\]
are represented in the quotient by
\[
 \beta_{\rm tr}=h_\beta h_\alpha,\qquad
 \beta_{\rm fig}=h_\alpha h_\beta^{-1}.
\]
These are mapping-class-level braid monodromy factorizations distinguishing
the figure-eight and trefoil blocks.  They do not assert that a particular
four-dimensional surgery presentation is equivariant.  That requires a
separate check of the surgery torus, its neighborhood, and the chosen
framing.

\subsection{\texorpdfstring{The relative knot-surgery description of $Y_K$}{The relative knot-surgery description of YK}}
\label{subsec:YK-piecewise}

Knot surgery on $T^2\times S^2$ along
$T^2\times\{\mathrm{pt}\}$ replaces $T^2\times D^2$ by
$S^1\times(S^3\setminus\operatorname{int}\nu K)$.  The product-circle
generator and the knot meridian are identified with the two generators of
the removed torus, while the preferred longitude is identified with its
normal circle; orientations are chosen so that the boundary gluing reverses
the normal orientation.  With this peripheral identification the result is
$M_K\times S^1$.
Performing the same replacement along a regular fiber of this manifold, with
the peripheral identification displayed below, gives the fiber sum that
identifies a fiber torus in one copy with a section torus in the other.  We
refer to this below as the fiber-section sum.
In this precise sense $Y_K$ has an iterated knot-surgery description.  This
is a reinterpretation of the construction of
\cite[Section~4]{Akh07}, not a different construction; the terminology is
that of \cite{FS}.

Following the peripheral conventions in
\cite[Section~4 and Lemma~4.6]{Akh07}, write
\[
 P_s=(M_K\times S^1)\setminus\operatorname{int}\nu S,
 \qquad
 P_f=(M_K\times S^1)\setminus\operatorname{int}\nu F,
\]
where $S$ is the distinguished section in the first copy and $F$ is a
regular fiber in the second.  Then
\[
 Y_K=P_s\cup_{\phi_1}P_f.
\]
The two copies carry independent quotient pairs
\[
 (T^2\times S^2,B_K^{(s)}),\qquad
 (T^2\times S^2,B_K^{(f)}),
\]
with
\[
 B_K^{(\ell)}=B_1^{(\ell)}\sqcup B_3^{(\ell)},
 \qquad \ell\in\{s,f\}.
\]
The component $B_1^{(\ell)}$ is traced by the zero two-torsion point, and
$B_3^{(\ell)}$ by the orbit of the three nonzero two-torsion points.  Both
blocks have monodromy $(\beta_K,1)$ on the two base circles.

Removing the zero section recovers the knot exterior product
\[
 P_s\cong(S^3\setminus\operatorname{int}\nu K)\times S^1.
\]
Its boundary marking is
\[
 (b,x,\lambda_K),
\]
where $b$ is the knot meridian, $x$ is the product-circle generator, and
$\lambda_K$ is the preferred longitude.  For the right-handed trefoil, let
$a,b$ be Wirtinger generators satisfying $aba=bab$, with $b$ the chosen
meridian.  Then
\[
 \lambda_K=ab^2ab^{-4}.
\]
We choose the fiberwise group structure so that the distinguished section
$S$ is the zero section.  It is therefore fixed pointwise by the elliptic
involution and maps to $B_1^{(s)}$.
Since $S$ is the ramification surface over $B_1^{(s)}$, restriction of
$\pi_K$ gives a relative double branched cover
\begin{equation}\label{eq:relative-section-cover}
 P_s\longrightarrow
 Q_s=(T^2\times S^2)\setminus\operatorname{int}\nu B_1^{(s)},
\end{equation}
branched along $B_3^{(s)}$.  On the boundary it is a connected free
two-fold cover.  In the equivariant product framing, let
$(\bar b,\bar x,\mu_1)$ be the two section directions and positive normal
meridian in the quotient.  Then
\[
 b\longmapsto\bar b,\qquad x\longmapsto\bar x,
 \qquad \lambda_K\longmapsto2\mu_1.
\]
The covering character
$\chi_s:\pi_1(Q_s\setminus B_3^{(s)})\to\mathbb Z/2$ restricts on the
boundary as
\[
 \chi_s(\bar b)=\chi_s(\bar x)=0,
 \qquad \chi_s(\mu_1)=1,
\]
and a meridian of $B_3^{(s)}$ also has value $1$.  This follows directly
from the local model $(u,w)\mapsto(u,w^2)$ along $B_1^{(s)}$.  Thus the
knot-exterior factor is a relative branched-cover model.  The class
$\lambda_K$ is oriented as the positive normal-circle class to the
ramification section and records the normal direction to the deleted section.  It is not
determined by the spherical braid.

For the fiber-complement block, the quotient of the deleted fiber is a
sphere.  Restricting over a punctured base leaves the punctured section
$B_1^{(f)}$ and punctured three-fold multisection $B_3^{(f)}$; their four boundary
circles are the traces of the four marked branch points.  With fiber
generators $(\gamma_1',\gamma_2')$ and base generators $d,y$, the boundary
marking is
\[
 (\gamma_1',\gamma_2',[d,y]),
\]
where $[d,y]$ is oriented as the negative of the positive normal meridian of
the deleted fiber.  With this convention, the following
orientation-reversing gluing sends the positive section meridian to
$[d,y]$, rather than to its inverse.

With the markings of the original construction, the orientation-reversing
boundary map satisfies
\[
 \phi_{1*}(x)=\gamma_1',\qquad
 \phi_{1*}(b)=\gamma_2',\qquad
 \phi_{1*}(\lambda_K)=[d,y].
\]
This is a peripheral map in the covering manifolds, not an identification of
quotient boundaries.  It records exactly the information absent from the
braid word.

The following description retains the data visible in the two local
quotients without presupposing a quotient map on the glued boundary.  On the
section-complement side the three sheets of $B_3^{(s)}$ determine the three
nonzero two-torsion labels.  On the fiber-complement side
$B_1^{(f)}\sqcup B_3^{(f)}$ determines all four branch-point labels.  The
map $\phi_1$ in the covering manifolds identifies the peripheral generators
according to
\[
 \begin{array}{c|c|c}
 \text{fiber-sum datum}&P_s&P_f\\ \hline
 \text{first surface circle}&x&\gamma_1'\\
 \text{second surface circle}&b&\gamma_2'\\
 \text{normal circle}&\lambda_K&[d,y].
 \end{array}
\]
Thus the two closed quotient pairs, the specified deletions in their covering
manifolds, and the map $\phi_{1*}$ form the marked gluing data for $Y_K$.
These data retain the braid labels and the normal-circle class, but do not
give an embedded banded branch surface in a common quotient.  Such a surface
would require a boundary diffeomorphism of the quotient manifolds whose lift
is $\phi_1$, which is the compatibility issue studied below.

The unused punctured fiber in $P_s$ and punctured section in $P_f$ join to
form a square-zero genus-two surface $\Sigma_2\subset Y_K$.  After choosing
the standard parallel push-off to $\partial\nu\Sigma_2$, its four surface
generators map into $Y_K\setminus\nu\Sigma_2$ as
\[
 u=a^{-1}b,\qquad v=b^{-1}aba^{-1},\qquad d,\qquad y,
\]
and a meridian is represented, up to basepoint conjugation, by
\[
 \mu_{\Sigma_2}=[x,b].
\]
For these identifications, including the preferred longitude and the
generators $u,v,d,y$, see \cite[Section~4.1, especially
Lemmas~4.3-4.7]{Akh07}.
Thus the description of $Y_K$ combines iterated knot surgery with explicit
braid monodromy and marked gluing data.  It contains the genuine relative cover
\eqref{eq:relative-section-cover}, the ambient quotient pair for the
fiber-complement block, the two peripheral markings, the map $\phi_1$, and
the marked surface $(\Sigma_2;u,v,d,y,\mu_{\Sigma_2})$.  The natural local
actions do not glue: the action on $\partial P_s$ is free, whereas the
fiberwise elliptic action on $\partial P_f$ has four fixed circles.
Consequently $\phi_1$ cannot conjugate these boundary actions.  A global
quotient requires a different relative action on one side and equality of
the restricted covering characters.

\subsection{\texorpdfstring{The twisted double $X_K$}{The twisted double XK}}
\label{subsec:XK-double}

The knot-indexed construction of \cite[Section~4]{Akh07} takes two copies
of $(Y_K,\Sigma_2)$ and glues their genus-two complements by the marked map
\[
 (m',x',\gamma_1',\gamma_2')\longmapsto
 (\gamma_1,\gamma_2,m,x),
 \qquad \mu'\longmapsto\mu^{-1}.
\]
Here $m,m'$ are the section-side knot meridians, $x,x'$ are the external
product-circle generators, and $\gamma_i,\gamma_i'$ are the marked fiber
generators in the notation of \cite{Akh07}.
With the product normal framings, this is a twisted double whose gluing has no
additional meridional shear; it is not the identity double.  The
Mayer--Vietoris calculation and the fiber-sum formulas in \cite{Akh07} give
\[
 H_1(X_K;\mathbb Z)=0,\qquad
 H_2(X_K;\mathbb Z)\cong\mathbb Z^2,
 \qquad e(X_K)=4,\qquad \sigma(X_K)=0.
\]
Denote the two surviving genus-two surfaces by $S$ and $T$.  Their classes
generate the hyperbolic intersection form.  The symplectic sum is symplectic
\cite{Gompf}, and the canonical-class calculation of
\cite[Section~4.2]{Akh07} gives
\[
 K_{X_K}=2\operatorname{PD}[S]+2\operatorname{PD}[T].
\]

For the present paper, $X_K$ is described by four local quotient pairs
$(T^2\times S^2,B_1\sqcup B_3)$ together with the two fiber-section gluing
maps defining the copies of $Y_K$ and the genus-two basis switch above.
These are explicit local branched-cover models with marked gluing data, not
a branched-cover description of the closed manifold.  The obstruction to
the two fiber-section gluings is proved in
Section~\ref{sec:boundary-obstruction}.

\section{Equivariant sums along transverse invariant surfaces}

The original construction forms
\[
 Y_K=(M_K\times S^1)\#_{F=S}(M_K\times S^1).
\]
The construction of $Y_K$ contains a fiber--section sum.  We first record
what happens when both summing surfaces meet the fixed sets transversely.

\begin{proposition}\label{prop:eqsum}
Let $(X_i,\iota_i)$, $i=1,2$, be smooth double branched covers, and suppose
$C_i\subset X_i$ are closed oriented invariant square-zero surfaces meeting
$\operatorname{Fix}(\iota_i)$ transversely.  Choose invariant product
neighborhoods $\nu C_i\cong C_i\times D^2$ in which
$\iota_i(z,w)=(\iota_i|_{C_i}(z),w)$.  Suppose the fiber-sum gluing
\[
 \Phi:\partial\nu C_1\longrightarrow\partial\nu C_2
\]
intertwines these product actions, reverses the oriented normal-circle
coordinate, and matches
the fixed boundary circles, their local intersection signs, the chosen
normal framings of $C_i$, and the local framings of the fixed surfaces at
$C_i\cap\operatorname{Fix}(\iota_i)$.  Then the fiber sum carries a natural involution and its
quotient is obtained by gluing the two quotient complements along the induced
boundary quotient map.  Equivariance also identifies the restricted covering
characters on the two quotient boundaries.  Near each transverse
intersection of a fixed surface with $C_i$, the branch surfaces are joined by
a local band attachment.
\end{proposition}

\begin{proof}
The equivariant tubular-neighborhood theorem gives the stated invariant
product models, and equivariant isotopy extension allows the marked boundary
identification to be realized on collars.  The equivariant boundary
identification then makes the two actions agree on the union, and taking the
quotient commutes with gluing.  At a transverse
intersection choose oriented complex coordinates with the summing surface
$C_i=\{w=0\}$, the fixed surface $\{z=0\}$, and involution
$(z,w)\mapsto(-z,w)$.  The quotient map is $(z,w)\mapsto(z^2,w)$; matching
these models and the chosen local framings of the fixed surfaces gives a
smooth quotient.  On
$|w_1|=|w_2|=\varepsilon$, the standard collar gluing may be written
$w_2=\varepsilon^2/w_1$.  Thus the fixed annuli glue smoothly, and in the
quotient coordinates $u_i=z_i^2$ the branch loci $u_i=0$ join by the asserted
an annulus joining the corresponding boundary circles of the branch surfaces.
\end{proof}

In the fiber--section sum, however, $F$ meets the fixed set transversely
whereas $S$ is a fixed component, with local model
\[
 (z,w)\longmapsto(z,-w),
\]
whose boundary action is free.  Lemma~\ref{lem:boundary-obstruction} rules
out an equivariant boundary identification, and
Theorem~\ref{thm:homology-product-obstruction} rules out an interpolation
whose lift is an integral homology product.

\section{The boundary obstruction and the mapping-class formulation}
\label{sec:boundary-obstruction}

The natural elliptic involution behaves differently on the two sides of the
fiber-section sum.  This prevents us from interpreting the original gluing as
an equivariant fiber sum and identifies the boundary data that any
alternative quotient construction would have to change.

On a tubular neighborhood
\[
 \nu F=T^2\times D^2
\]
of a torus fiber, the involution has the form
\[
 (z,w)\longmapsto(-z,w).
\]
Hence on
\[
 H_1(\partial\nu F;\mathbb Z)
 =
 \langle\gamma_1,\gamma_2,\mu_F\rangle
\]
it acts by
\[
 I_F=
 \begin{pmatrix}
 -1&0&0\\
 0&-1&0\\
 0&0&1
 \end{pmatrix}.
\]
Its fixed set on the boundary consists of four circles.

For a fixed section $T_m$, the local model is
\[
 (z,w)\longmapsto(z,-w).
\]
The induced action on the boundary is free, since $w\mapsto-w$ is a half
rotation of the normal circle.  Its action on first homology is the identity.
Thus the two boundary actions are not conjugate.

\begin{lemma}\label{lem:boundary-obstruction}
The natural elliptic involutions on a fiber neighborhood and a fixed-section
neighborhood cannot be identified by an equivariant boundary
diffeomorphism.
\end{lemma}

\begin{proof}
The fixed set of the fiber-boundary action is four circles, whereas the
section-boundary action is free.  Thus the actions cannot be equivariantly
diffeomorphic.  The same conclusion follows on first homology, since the
fiber-boundary action has eigenvalues $(-1,-1,1)$ while the section-boundary
action is the identity.
\end{proof}

The quotient boundaries are correspondingly different:
\[
 \partial\nu F/\iota_F\cong S^2\times S^1,
 \qquad
 \partial\nu T_m/\iota_S\cong T^2\times S^1.
\]
The first quotient carries four branch circles, while the second boundary is
unbranched.  Any interpolation in the quotient would therefore require a
branched cobordism rather than an equivariant fiber sum.

\begin{definition}
A \emph{fiber-section branch cobordism} is a triple $(C,\mathcal A,q)$,
where $C$ is a compact oriented four-manifold from $S^2\times S^1$ to
$T^2\times S^1$, $\mathcal A\subset C$ is a properly embedded surface
transverse to the boundary and product near the boundary, and
\[
 \mathcal A\cap\partial_-C=\bigsqcup_{i=1}^{4}S_i^1,
 \qquad
 \mathcal A\cap\partial_+C=\varnothing.
\]
$q:\widetilde C\to C$ is a specified smooth double cover branched along
$\mathcal A$.  Fix boundary parameterizations by the displayed products.
The cover is required to have
boundary $T^3\sqcup T^3$ and to restrict, relative to these markings, to the
elliptic fiber cover on the incoming boundary and the free section-boundary
cover on the outgoing boundary.
\end{definition}

Let $m_1,\ldots,m_4$ be small meridians of the branch circles and let $n$ be
the product circle on the fiber quotient boundary.  The incoming character
\[
 \chi_F:\pi_1\left((S^2\times S^1)\setminus
 \nu\left(\bigsqcup_{i=1}^4S_i^1\right)\right)\longrightarrow\mathbb Z/2
\]
satisfies
\[
 \chi_F(m_i)=1,\qquad \chi_F(n)=0.
\]
In
$H_1((S^2\times S^1)\setminus\nu(\bigsqcup S_i^1);\mathbb Z/2)$ one has
$m_1+\cdots+m_4=0$, so these values are consistent.  On the section quotient boundary, writing $a,b$ for the torus
generators and $n_S$ for the quotient normal circle, the free cover has a
character $\chi_S:\pi_1(T^2\times S^1)\to\mathbb Z/2$ satisfying
\[
 \chi_S(a)=\chi_S(b)=0,\qquad \chi_S(n_S)=1.
\]
The branched cover determines a character
$\chi_C:\pi_1(C\setminus\mathcal A)\to\mathbb Z/2$, taking value one on
every branch meridian, whose restrictions under the chosen boundary
parameterizations are $\chi_F$ and $\chi_S$.
Along $\mathcal A$ it is required to have the standard meridional local
branch model.

Call $\widetilde C$ an \emph{integral homology product} if both boundary
inclusions induce isomorphisms on integral homology.

\begin{theorem}\label{thm:homology-product-obstruction}
Let $(C,\mathcal A,q)$ be a fiber-section branch cobordism, with deck
transformation $\tau$.  The two boundary inclusions $T^3\to\widetilde C$
cannot both induce isomorphisms on $H_1(-;\mathbb Z)$.  In particular,
$\widetilde C$ is neither an integral homology product nor a product
cobordism $T^3\times[0,1]$ relative to its boundary.
\end{theorem}

\begin{proof}
Suppose that the two boundary inclusions $j_-$ and $j_+$ induce isomorphisms
on $H_1$.  Equivariance gives
\[
 (j_\pm)_*\tau_{\pm *}=\tau_*(j_\pm)_*.
\]
The isomorphism
$\Phi=(j_+)_*^{-1}(j_-)_*$ then satisfies
$\Phi\tau_{-*}=\tau_{+*}\Phi$, so the two boundary actions are conjugate.  At the
fiber end, in the basis $(\gamma_1,\gamma_2,\mu_F)$, the action is
\[
 \tau_{-*}=\operatorname{diag}(-1,-1,1).
\]
At the fixed-section end, the deck transformation is a half-translation in
the normal circle and acts trivially on $H_1(T^3)$, so
\[
 \tau_{+*}=I_3.
\]
These matrices are not conjugate over $\mathbb Z$; their traces are $-1$ and
$3$.  This is a contradiction.
\end{proof}

The long exact sequences of the boundary pairs show more precisely that
$H_k(\widetilde C,\partial_\varepsilon\widetilde C;\mathbb Z)\ne0$ for some
$\varepsilon\in\{-,+\}$ and $k\in\{1,2\}$.

\begin{corollary}\label{cor:no-collar-replacement}
There is no fiber-section branch cobordism whose specified double cover is a
collar $T^3\times[0,1]$ relative to the prescribed boundary markings.
\end{corollary}

\begin{proof}
Both boundary inclusions of a collar $T^3\times[0,1]$ are homotopy
equivalences, contradicting
Theorem~\ref{thm:homology-product-obstruction}.
\end{proof}

The original construction determines the marking which this cobordism must
realize.  With the notation of the trefoil calculation, the boundary
identifications include
\[
 \phi_{1*}(x)=\gamma_1',
 \qquad
 \phi_{1*}(b)=\gamma_2',
\]
and the remaining peripheral circle identifies the knot longitude with the
negatively oriented meridian $\mu_F=[d,y]$ of the removed torus fiber.  Thus
the desired lifted marking, with the orientation convention fixed above, is
\[
 (b,x,\lambda_K)
 \longmapsto
 (\gamma_2',\gamma_1',\mu_F).
\]

The braid words of Section~\ref{sec:piecewise-braids} record the branch-point
motion along the mapping-torus
generator.  The monodromy along the second base circle is trivial, and the
word contains no information about the normal-circle class.  More importantly,
Theorem~\ref{thm:homology-product-obstruction} shows that even a prescribed
identification
\[
 (b,x,\lambda_K)\longmapsto(\gamma_2',\gamma_1',\mu_F)
\]
cannot turn the lifted cobordism into the collar removed in the original
fiber-section sum.  Any viable replacement must have nonzero relative
homology at least at one end, and this topology must be accounted for in the
closed manifold.  Thus
the preceding discussion supplies necessary boundary constraints, not an
existence theorem for a quotient of $Y_K$.

\section{The Matsumoto branch configuration}
\label{sec:matsumoto}

For the $X_5$ construction, choose distinct
points $p_1,p_2\in T^2$ and $q_1,q_2\in S^2$, and put
\[
 H_i=S^2\times\{p_i\},\qquad V_j=\{q_j\}\times T^2,
 \qquad i,j\in\{1,2\}.
\]
Let
\[
 A=[S^2\times\{\mathrm{pt}\}],\qquad
 C=[\{\mathrm{pt}\}\times T^2].
\]
Thus $A^2=C^2=0$ and $A\cdot C=1$.  The divisor
\[
D=H_1\cup H_2\cup V_1\cup V_2
\]
has exactly four transverse double points $(q_j,p_i)$.

\begin{figure}[H]
\centering
\begin{tikzpicture}[>=Latex,font=\scriptsize]
 \node[rounded corners=7pt,fill=white,draw=branchblue!25,
   blur shadow={shadow opacity=.17,shadow blur steps=5},
   minimum width=5.75cm,minimum height=3.55cm] at (0,0) {};
 \node[rounded corners=7pt,fill=white,draw=turquoise!30,
   blur shadow={shadow opacity=.17,shadow blur steps=5},
   minimum width=5.05cm,minimum height=3.55cm] at (7.2,0) {};
 \filldraw[rounded corners=5pt,fill=branchblue!4,draw=branchblue!45,
   line width=.9pt] (-2.6,-1.35) rectangle (2.6,1.35);
 \node[text=branchblue!80!black] at (0,1.22)
   {$S^2\times T^2$};
 \draw[branchred,line width=2.15pt] (-2.25,.62)--(2.25,.62)
   node[right,text=branchred] {$H_1$};
 \draw[branchblue,line width=2.15pt] (-2.25,-.62)--(2.25,-.62)
   node[right,text=branchblue] {$H_2$};
 \draw[gluegreen,line width=2.15pt] (-.82,-1.12)--(-.82,1.12)
   node[above,text=gluegreen] {$V_1$};
 \draw[turquoise,line width=2.15pt] (.82,-1.12)--(.82,1.12)
   node[above,text=turquoise] {$V_2$};
 \foreach \x in {-.82,.82}{\foreach \y in {-.62,.62}{
   \filldraw[fill=warmorange!85,draw=blockpurple,line width=.9pt]
     (\x,\y) circle (.17);
 }}
 \draw[-{Latex[length=2.8mm]},blockpurple,line width=1.45pt]
   (3.15,0)--node[above,fill=white,inner sep=1pt,text=blockpurple!80!black]
   {blow up} (4.35,0);
 \begin{scope}[xshift=7.2cm]
  \filldraw[rounded corners=5pt,fill=turquoise!5,draw=turquoise!50,
    line width=.9pt] (-2.25,-1.35) rectangle (2.25,1.35);
  \draw[branchred,line width=2.15pt] (-1.9,.72)--(-.25,.72);
  \draw[branchred,line width=2.15pt] (.25,.72)--(1.9,.72)
    node[right,text=branchred] {$\widetilde H_1$};
  \draw[branchblue,line width=2.15pt] (-1.9,-.72)--(-.25,-.72);
  \draw[branchblue,line width=2.15pt] (.25,-.72)--(1.9,-.72)
    node[right,text=branchblue] {$\widetilde H_2$};
  \draw[gluegreen,line width=2.15pt] (-.72,-1.12)--(-.72,-.96);
  \draw[gluegreen,line width=2.15pt] (-.72,-.48)--(-.72,.48);
  \draw[gluegreen,line width=2.15pt] (-.72,.96)--(-.72,1.12);
  \draw[turquoise,line width=2.15pt] (.72,-1.12)--(.72,-.96);
  \draw[turquoise,line width=2.15pt] (.72,-.48)--(.72,.48);
  \draw[turquoise,line width=2.15pt] (.72,.96)--(.72,1.12);
  \foreach \x in {-.72,.72}{\foreach \y in {-.72,.72}{
    \filldraw[fill=blockpurple!22,draw=blockpurple,line width=1.05pt]
      (\x,\y) circle (.24);
  }}
  \node[text=blockpurple] at (0,-1.58) {exceptional curves $E_{ij}$};
 \end{scope}
\end{tikzpicture}
\caption{The Matsumoto branch divisor before and after the four blowups.
Red and blue denote the horizontal components, green and turquoise the
vertical components, and purple the exceptional curves.  The proper
transforms are disjoint.  The picture is schematic; incidences of each
exceptional curve with the adjacent proper transforms are suppressed.}
\label{fig:matsumoto-divisor}
\end{figure}
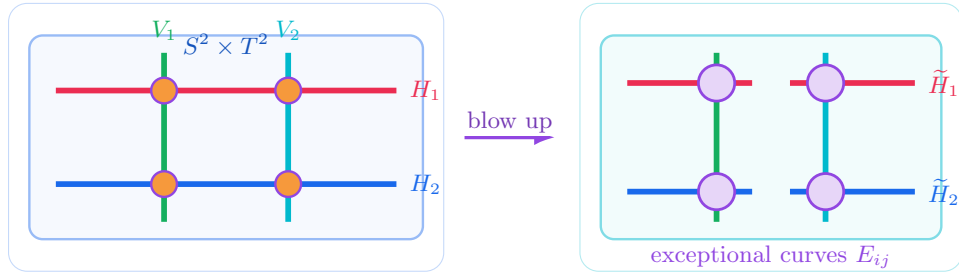

Let
\[
 \rho:W=(S^2\times T^2)\#4\CPb^2\longrightarrow S^2\times T^2
\]
be the blowup at these points and let $E_{ij}$ be the exceptional curves.
The proper transforms have classes
\[
 [\widetilde H_i]=A-E_{i1}-E_{i2},\qquad
 [\widetilde V_j]=C-E_{1j}-E_{2j}.
\]
They are pairwise disjoint and each has square $-2$.  Their union
\[
 B=\widetilde H_1\sqcup\widetilde H_2\sqcup
   \widetilde V_1\sqcup\widetilde V_2
\]
satisfies
\[
 [B]=2A+2C-2\sum_{i,j=1}^2E_{ij}.
\]
Put
\[
 \ell=\operatorname{PD}\left(A+C-\sum_{i,j=1}^2E_{ij}\right)
 \in H^2(W;\mathbb Z).
\]
Then $\operatorname{PD}[B]=2\ell$.  Let $E_B\to W$ be the complex line
bundle associated with the oriented codimension-two submanifold $B$, so
$c_1(E_B)=\operatorname{PD}[B]$, and choose a transverse section $s_B$ of
$E_B$ whose zero set is $B$.  Choose a complex line bundle
$\mathcal L\to W$ with $c_1(\mathcal L)=\ell$ and an isomorphism
$\mathcal L^{\otimes2}\cong E_B$.  The associated branched cover is
\[
 Z=\{v\in\mathcal L\mid v^{\otimes2}=s_B(\pi(v))\}.
\]
Once $c_1(\mathcal L)=\ell$ is fixed, the remaining double-cover data are
encoded by the covering character on $W\setminus B$; possible characters
differ by classes pulled back from $H^1(W;\mathbb Z/2)$.  We choose the
character induced by the classical Matsumoto involution.  The standard
square-root construction gives the
smooth double cover
\[
 p:Z\longrightarrow W
\]
branched along $B$.  This is a resolved double-cover model on the unmarked
total space of the Matsumoto fibration, and
\[
 Z\cong (T^2\times S^2)\#4\CPb^2
\]
\cite{Matsumoto,AkhmedovSmall}.  Near a double point of $D$, the unresolved
cover has equation $z^2=uv$.  Blowing up the crossing in the quotient and taking
the cover branched along the disjoint proper transforms gives the minimal
resolution of this $A_1$ singularity.

\subsection{The covering character and the four resolutions}

Let $U=W\setminus B$.  Choose the $\mathbb Z/2$ monodromy homomorphism
associated with the Matsumoto double cover,
\[
 \chi_{\rm Mat}:\pi_1(U)\longrightarrow\mathbb Z/2
\]
which takes a positively oriented meridian of every component of $B$ to
$1$.  Its existence is supplied by the classical Matsumoto involution.
Thus the unbranched cover of $U$ determined by
$\ker\chi_{\rm Mat}$ extends uniquely, once $\chi_{\rm Mat}$ is fixed,
across the meridional double covers of
the normal disks.  This gives the specified branched cover $p:Z\to W$.
The choice of character is part of the marked quotient data; the homology
class of $B$ alone does not specify it when the base has nontrivial first
cohomology.

We compare the unresolved and resolved local models.  Near a
crossing of $H_i$ and $V_j$, choose complex coordinates so that the two
components are $\{u=0\}$ and $\{v=0\}$.  The double cover is
\[
 z^2=uv.
\]
The branch divisor has a transverse node, and its double cover has the
$A_1$ singularity $z^2=uv$.  Blow up the origin in the quotient.  In the chart
$v=ut$, the total transform of the branch divisor is
$u^2t=0$.  The reduced branch divisor is the strict transform $t=0$; the
exceptional curve $u=0$ occurs with even multiplicity and is therefore not
a component of the reduced branch divisor.  After removing this even
exceptional multiplicity, the normalized
cover has local equation $z_1^2=t$, where $z=uz_1$.  The second chart is
identical.  Hence the normalized cover is smooth, and the inverse image of
the exceptional curve is a sphere meeting the two ramification components.
Repeating this at the four crossings produces the four spheres $Q_{ij}$
described below.

The global class calculation is compatible with this local normalization.
If
\[
 L=A+C-\sum_{i,j}E_{ij},
\]
then $[B]=2L$.  Since
\[
 A^2=C^2=0,\qquad A\cdot C=1,\qquad E_{ij}^2=-1,
\]
one obtains
\[
 L^2=2-4=-2,\qquad B^2=4L^2=-8.
\]
\subsection{Two quotient viewpoints on the genus-two fiber}

The map $f=\operatorname{pr}_1\circ\rho\circ p$ is the resolved,
unperturbed double-cover map on the Matsumoto total space.
Along a vertical ramification component its local form is a branched map
$u=z^2$, so it is not yet a Lefschetz fibration.  The four nodes belong to
the unresolved divisor in $S^2\times T^2$; after blowing up, the proper
transforms in $W$ are disjoint.  Smoothing the unresolved nodal divisor and
normalizing its cover is distinct from perturbing the resolved map near its
vertical ramification.  The standard local comparison gives the usual
Matsumoto Lefschetz fibration on a diffeomorphic unmarked total space
\cite{Matsumoto,AkhmedovSmall}.  No marked-complement identification is used
below.  Over a regular value $q$, the fiber
of the unperturbed model is a double cover
of $T^2$ branched at the two points $p_1,p_2$.  The deck involution of this
cover will be denoted by $\tau_{\rm bi}$.  Its quotient is a torus and its
fixed set on the genus-two fiber consists of two points.

The same genus-two surface also has its hyperelliptic involution
$\tau_{\rm hyp}$, whose quotient is $S^2$ and whose fixed set consists of six
points.  These involutions play different roles.  The divisor picture in
Figure~\ref{fig:matsumoto-divisor} realizes $\tau_{\rm bi}$ and is adapted to
the horizontal and vertical components of $B$.  A possible extension across
the genus-two sum in Section~\ref{sec:global-problem} would instead use
$\tau_{\rm hyp}$ on both sides of the sum.  The two quotient pictures should
not be identified: their quotient fibers and boundary branch links are
different.

On a boundary $\Sigma_2'\times S^1$ of a regular fiber neighborhood, the
bielliptic quotient is $T^2\times S^1$ with two branch circles.  The
hyperelliptic quotient is $S^2\times S^1$ with six branch circles.  In either
case the product involution acts trivially on the normal-circle factor, but
the surface part of a gluing map and any normal-circle shear must be checked
separately.  The six-circle picture is a local quotient of the marked fiber
boundary; it is not the restriction of the established cover $p:Z\to W$,
whose fiberwise deck action is bielliptic.

The branched-cover formulas give
\[
 e(W)=4,\qquad \sigma(W)=-4,\qquad
 e(B)=2+2+0+0=4,\qquad B^2=-8.
\]
Consequently
\[
 e(Z)=2e(W)-e(B)=4,\qquad
 \sigma(Z)=2\sigma(W)-\frac12B^2=-4,
\]
as required for $(T^2\times S^2)\#4\CPb^2$.

The curves needed below are as follows.  Put
\[
 Q_{ij}=p^{-1}(E_{ij}).
\]
Since $E_{ij}$ meets $B$ twice, $Q_{ij}\to E_{ij}$ is the connected double
cover of $S^2$ branched at two points.  Hence $Q_{ij}$ is a sphere and
\[
 p^*E_{ij}=Q_{ij},\qquad
 Q_{ij}^2=(p^*E_{ij})^2=2E_{ij}^2=-2.
\]
It is not a ramification component and is not an exceptional sphere.
Second, if $S_i$ and $T_j$ denote the reduced inverse images of
$\widetilde H_i$ and $\widetilde V_j$, then
\[
 p^*\widetilde H_i=2S_i,\qquad p^*\widetilde V_j=2T_j,
 \qquad S_i^2=T_j^2=-1.
\]
Here $S_i$ is a sphere and an exceptional section of the genus-two
fibration, whereas $T_j$ is a genus-one curve contained in a singular fiber
of the unperturbed resolved model.
Thus $T_j$ has square $-1$ but is not an exceptional sphere.

For the unperturbed branched model, consider
\[
 f=\operatorname{pr}_1\circ\rho\circ p:Z\longrightarrow S^2.
\]
Relative to the projection $\operatorname{pr}_1$, the components $H_1,H_2$
are sections and $V_1,V_2$ are vertical fibers.  Thus $D$ is a braided
branch configuration consisting of two horizontal sheets and two vertical
components.  This is not a degree-four braided
multisection: the vertical components lie over the two exceptional base
values $q_1,q_2$.  The blowups resolve the four crossings, and the marked
pair $(W,B)$ retains this horizontal-vertical decomposition.
For $q\notin\{q_1,q_2\}$, the quotient fiber $\{q\}\times T^2$ meets the
branch divisor at $(q,p_1)$ and $(q,p_2)$.  Its inverse image
$\Sigma_2'$ is the double cover of $T^2$ branched at these two points.
Riemann-Hurwitz gives
\[
 2g(\Sigma_2')-2=2\bigl(2g(T^2)-2\bigr)+2=2,
\]
so $g(\Sigma_2')=2$.  The curves $S_1$ and $S_2$ meet every regular fiber
once and give two disjoint exceptional sections.

On the boundary of a regular fiber neighborhood, the covering is
\[
 (\Sigma_2'\times S^1,\tau_{\Sigma}\times\id)
 \longrightarrow T^2\times S^1,
\]
where $\tau_{\Sigma}=\tau_{\rm bi}$ is the deck involution of
$\Sigma_2'\to T^2$.  It has two fixed points, so the fixed set on the
boundary consists of two circles $L_1\sqcup L_2$.  If $m_1,m_2$ are the
corresponding branch meridians and $n$ is the normal-circle class, the
covering character
\[
 \chi:\pi_1\left((T^2\times S^1)\setminus
 \nu(L_1\sqcup L_2)\right)\longrightarrow\mathbb Z/2
\]
satisfies, writing $A_0,B_0$ for marked generators of the quotient torus,
\[
 \chi(A_0)=\varepsilon_A,\qquad \chi(B_0)=\varepsilon_B,\qquad
 \chi(m_1)=\chi(m_2)=1,\qquad \chi(n)=0
 \quad\text{in }\mathbb Z/2.
\]
Here $(\varepsilon_A,\varepsilon_B)$ is part of the chosen Matsumoto
covering character; the three meridional and normal-circle values alone do
not determine the cover.
A genus-two boundary gluing descends only if it conjugates the deck actions
and preserves the branch circles, the covering character, and the
normal-circle class.  Agreement on first homology alone does not imply these
marked conditions.

The complement calculation used in Section~\ref{sec:x5} is
A punctured exceptional section is a disk bounded by a meridian of
$\Sigma_2'$, so that meridian is nullhomotopic in
$Z\setminus\nu\Sigma_2'$.  With standard generators
$a_1,b_1,a_2,b_2$ of a regular fiber,
\[
 \pi_1(Z\setminus\nu\Sigma_2')=
 \left\langle a_1,b_1,a_2,b_2\ \middle|\
 b_1b_2,\ [a_1,b_1],\ [a_2,b_2],\
 b_2a_2b_2^{-1}a_1
 \right\rangle
 \cong\mathbb Z^2
\]
\cite[Sections~3 and~5]{AkhmedovSmall}.  Indeed, $b_2=b_1^{-1}$, the last
relation expresses $a_1$ in terms of $a_2,b_1$, and the commutator relations
give $[a_2,b_1]=1$ and $a_1=a_2^{-1}$.  Thus $a_1a_2=1$ is derived rather
than imposed.  These relations and the trivial meridian are the
Matsumoto-side input to the van Kampen calculation below.

\section{\texorpdfstring{The $\CP^2\#5\CPb^2$ example}{A manifold homeomorphic to CP2\#5CPbar2}}
\label{sec:x5}

For this section $K$ is the right-handed trefoil.  This is the case for which
the marked complement presentations were computed in
\cite[Sections~3 and~5]{AkhmedovSmall}.  Let
\[
 Z=(T^2\times S^2)\#4\CPb^2
\]
and let $\Sigma_2'\subset Z$ be the regular Matsumoto fiber described in
Section~\ref{sec:matsumoto}.  The first presentation is
\[
 X_5=Y_K\#_{\Sigma_2=\Sigma_2'}Z.
\]
This is the presentation in \cite{AkhmedovSmall}.  We also discuss the
symmetric construction scheme of \cite[Section~3.3]{ABP}.  We denote its
output by $X_5^{\rm sym}$.  The cited sources show that $X_5$ and
$X_5^{\rm sym}$ have the same homeomorphism type and the same numerical
invariants, but do not identify them diffeomorphically.

\subsection{The symmetric twisted double}

After moving the four Luttinger surgeries into the two summands, as in the
Appendix of \cite{ABP}, the symmetric construction of
\cite[Section~3.3]{ABP} can be viewed as using the genus-two surface obtained
by smoothing a fiber $F$ and a section $S$ in $M_K\times S^1$ and blowing
up two points:
\[
 V_K=(M_K\times S^1)\#2\CPb^2,\qquad
 \Gamma_K=[F]+[S]-E_1-E_2,\qquad \Gamma_K^2=0.
\]
It expresses $X_5^{\rm sym}$ as the twisted double of
$V_K\setminus\nu\Gamma_K$ by the basis switch
\[
 (a,b,c,d)\longmapsto(c',d',a',b'),
 \qquad \mu\longmapsto(\mu')^{-1},
\]
 with the product-framed gluing and no additional meridional shear.
Independently, equivariantly blowing up a free orbit of two points gives a
elliptic involution acts freely gives a branched cover
\[
 V_K\longrightarrow(T^2\times S^2)\#\CPb^2
\]
with the same branch surface $B_1\sqcup B_3$.  It is not a cover of the
marked pair: near $F\cap S$, the involution
$(u,v)\mapsto(-u,v)$ exchanges the smoothings
$uv=\varepsilon$ and $uv=-\varepsilon$.

\begin{proposition}\label{prop:x5-abp}
Each of the two summands in the twisted-double construction of
$X_5^{\rm sym}$ admits a local branched-cover model, but the natural
elliptic involution does not preserve the connected smoothing defining
$\Gamma_K$.  Consequently, the natural local involutions do not assemble
through this marked twisted-double construction.
\end{proposition}

\begin{proof}
The twisted-double decomposition is \cite[Section~3.3]{ABP}; the ambient
cover follows by equivariantly blowing up a free orbit in
Proposition~\ref{prop:MKcover}.  The local formula above proves the
smoothing obstruction.
\end{proof}

\begin{proposition}\label{prop:x5-decorated}
The construction of $X_5$ is a fiber sum of three summands carrying local
branched-cover models.  Two are copies of
\[
 (M_K\times S^1,\iota_K)\longrightarrow
 (T^2\times S^2,B_1\sqcup B_3),
\]
where $B_1$ is a section, $B_3$ is a connected three-fold multisection, and the
nontrivial spherical braid monodromy is represented by
$\beta_{\rm tr}=h_\beta h_\alpha$.  The third is the Matsumoto
double cover
\[
 (Z,\tau_Z)\longrightarrow (W,B),\qquad
 W=(S^2\times T^2)\#4\CPb^2.
\]
Here $\tau_Z$ is the deck involution of $p:Z\to W$.
The marked fiber-section gluing $\phi_1$ defines $Y_K$, and the marked
genus-two gluing $\phi_2$ defines $X_5$.
\end{proposition}

\begin{proof}
The first two quotient pairs follow from
Proposition~\ref{prop:MKcover} and the splitting of the branch surface in
Proposition~\ref{prop:split}.  Section~\ref{sec:matsumoto} supplies the third
marked pair, including its horizontal and vertical branch components.  The
boundary identifications displayed below are the two gluing maps in the
covering manifolds.  No equivariance of these gluing maps is asserted.
\end{proof}

\subsection{The two braided pieces and the first gluing}

Write $P_s$ for the copy of $M_K\times S^1$ from which the section is removed
and $P_f$ for the copy from which the fiber is removed.  Associate to them
separate quotient pairs
\[
 (T^2\times S^2,B_{\rm tr}^{(s)}),\qquad
 (T^2\times S^2,B_{\rm tr}^{(f)}),
\]
where
\[
 B_{\rm tr}^{(\ell)}=B_1^{(\ell)}\sqcup B_3^{(\ell)},
 \qquad \ell\in\{s,f\},
\]
and the monodromy along the nontrivial base circle is represented by
$\beta_{\rm tr}=h_\beta h_\alpha$.  These are two local quotient models
with marked boundary data; they are not quotient complements already glued
in a common quotient manifold.

With the markings of \cite[Section~3.2.2]{AkhmedovSmall}, orient the first
boundary map from the section-complement side to the fiber-complement side.
Its induced map on the peripheral generators is
\[
 \phi_{1*}(x)=\gamma_1',\qquad
 \phi_{1*}(b)=\gamma_2',\qquad
 \phi_{1*}(\lambda_K)=[d,y],
\]
where $b$ is the knot meridian, $x$ is the additional product circle, and
\[
 \lambda_K=ab^2ab^{-4}
\]
is the preferred trefoil longitude.  In the convention fixed in
Section~\ref{subsec:YK-piecewise}, the last target is the negatively oriented
meridian of the deleted fiber.  This is marked peripheral data in the
covering manifolds.  By
Lemma~\ref{lem:boundary-obstruction}, it is not an equivariant
identification of the natural quotient boundaries.

\subsection{The genus-two sum and simple connectivity}

The unused section in $P_s$ and the unused fiber in $P_f$ join to form a
square-zero symplectic genus-two surface $\Sigma_2\subset Y_K$.  With the
conventions of \cite[Section~3.2]{AkhmedovSmall}, its standard generators
have images
\[
 u=a^{-1}b,\qquad v=b^{-1}aba^{-1},\qquad d,\qquad y
\]
in the complement.  A meridian $\mu$ of $\Sigma_2$ is represented, up to
basepoint paths, by $[x,b]$.

On the Matsumoto side use the generators $(a_1,b_1,a_2,b_2)$ fixed in
Section~\ref{sec:matsumoto}.  The second boundary map is chosen so that
\[
\begin{aligned}
 \phi_{2*}(u)&=a_2,& \phi_{2*}(v)&=b_2,&
 \phi_{2*}(d)&=a_1,\\
 \phi_{2*}(y)&=b_1,& \phi_{2*}(\mu)&=(\mu')^{-1}.&&
\end{aligned}
\]
Using the product normal framings, we extend the displayed surface map with
no additional meridional shear.  The inverse records reversal of the
oriented normal circle and is immaterial
to the group calculation because $\mu'$ is nullhomotopic.
The meridian $\mu'$ is nullhomotopic in
$Z\setminus\nu\Sigma_2'$, since it bounds a punctured exceptional section.
Thus $\mu$ becomes trivial after the sum.

In van Kampen form, $\pi_1(X_5)$ is the pushout of the two complement
groups over $\pi_1(\Sigma_2\times S^1)$ with the Matsumoto boundary map
precomposed by $\phi_{2*}$.  Since $\mu'$ bounds a punctured exceptional
section, the normal generator dies first.  After the four surface-generator
identifications, the Matsumoto complement relations become
\[
 ud=1,\qquad yv=1,\qquad [u,v]=[d,y]=1.
\]
After $\mu$ is killed, we may use the presentation of $\pi_1(Y_K)$ from
\cite[Section~3.2]{AkhmedovSmall}.  Besides
$[y,b]=1$, $aba=bab$, $dbd^{-1}=x^{-1}$, and
$dxd^{-1}=xb$, the Matsumoto relations give
\[
 a^{-1}bd=1,\qquad
 yb^{-1}aba^{-1}=1,\qquad
 [a^{-1}b,b^{-1}aba^{-1}]=1,\qquad [d,y]=1.
\]
The first equality gives $a=bd$, while the second, using $[y,b]=1$, gives
$yab=ba$.  Multiplying this equality on the right by $a$ and using
$aba=bab$ yields $yab=a^2$.  Comparing again with $yab=ba$ gives
$ba=a^2$, hence $b=a$.  Now $a=bd$ gives $d=1$;
$dbd^{-1}=x^{-1}$ and $dxd^{-1}=xb$ then give $b=x=a=1$, and
$yb^{-1}aba^{-1}=1$ gives $y=1$.  Therefore $u,v,d,y$ are trivial, and so
are the identified generators $a_1,b_1,a_2,b_2$ of the Matsumoto
complement.  Thus
\[
 \pi_1(X_5)=1.
\]
The displayed boundary map is the datum that a braided replacement must
preserve.  This is the reduction used in
\cite[Lemma~5.1]{AkhmedovSmall}, referring there to the proof of Lemma~4.1.

\subsection{Topology, minimality, and exoticity}

The sum formulas give
\[
 e(X_5)=8,\qquad \sigma(X_5)=-4.
\]
Together with simple connectivity these invariants give a unimodular form of
rank six and signature $-4$.  It is odd, since an even unimodular form has
signature divisible by eight, and hence it is
$\langle1\rangle\oplus5\langle-1\rangle$.  Therefore
$X_5\approx\CP^2\#5\CPb^2$ by Freedman's theorem \cite{Freedman}; here
$b_1=0$, $(b_2^+,b_2^-)=(1,5)$, and the Kirby--Siebenmann invariant vanishes
because $X_5$ is smooth.  The minimality argument in
\cite[Section~5]{AkhmedovSmall} applies Usher's theorem \cite{Usher}:
$Y_K$ is minimal, and the four exceptional sections in $Z$ meet the summing
surface.  Hence $X_5$ is symplectically minimal.  Since $\pi_1(X_5)=1$,
Hamilton--Kotschick \cite{HamiltonKotschick} implies that $X_5$ is
irreducible.  The standard $\CP^2\#5\CPb^2$ is reducible, so the two
manifolds are not diffeomorphic.  These facts identify the covering
manifold but do not supply a branched-cover description of $X_5$.
Theorem~\ref{thm:homology-product-obstruction} shows that the natural
fiber-section quotient cannot be completed by band attachments in a product
collar.  A global cover, if one exists, must use a
different involution on $Y_K$ or a non-product replacement whose additional
handles are included explicitly.

The covering formulas suggest the following numerical possibility.  Keep the Matsumoto
quotient
\[
 W=(S^2\times T^2)\#4\CPb^2
\]
and seek a branch surface $B_5\subset W$ homologous to the Matsumoto branch
divisor but obtained from it by two orientable tube additions in disjoint
four-balls, carrying the trefoil braid data.  The interiors of the tubing
arcs are disjoint from the divisor, and at each attaching disk we choose the
untwisted local framing relative to a normal trivialization there.  The
Matsumoto divisor has Euler characteristic $4$ and
self-intersection $-8$.  Two tubes would give
\[
 e(B_5)=0,\qquad B_5^2=-8,
\]
and the corresponding double cover would have
\[
 e=2e(W)-e(B_5)=8,\qquad
 \sigma=2\sigma(W)-\frac12B_5^2=-4.
\]
These are exactly the invariants of $X_5$.  Ordinary tube additions produce
an embedded surface with these numerical invariants.  What is not
constructed here is a tubed surface carrying the required trefoil braid,
covering character, and marked gluing data.  The torus-bundle quotient has four
branch points in a generic spherical fiber, whereas the Matsumoto divisor
has two generic horizontal sheets and two vertical components.  A genuine
construction must supply a transition region creating or absorbing the
additional branch points, satisfy the global braid relation, compute the
covering fundamental group, and prove that the lifted modification realizes
the marked gluings $\phi_1$ and $\phi_2$.

\section{\texorpdfstring{The $\CP^2\#3\CPb^2$ example}{The CP2 blowup at three points}}
\label{sec:x3}

The $X_3$ construction replaces the Matsumoto block by a third torus-bundle
block.  Throughout this section $K$ is the right-handed trefoil, and we use
the markings of \cite[Sections~3 and~5]{AP2008}.  Put
\[
 Q_K=(M_K\times S^1)\#2\CPb^2.
\]
Although the marked summing surface used below is not invariant, the ambient
manifold $Q_K$ itself has a branched-cover description.

\begin{proposition}\label{prop:QK-cover}
For either genus-one fibered knot considered here, $Q_K$ is a two-fold
branched cover of
\[
 (T^2\times S^2)\#\CPb^2.
\]
The branch surface is the proper transform of $B_1\sqcup B_3$, disjoint from
the exceptional divisor.
\end{proposition}

\begin{proof}
Choose an orbit of two points in $M_K\times S^1$, disjoint from the fixed
set and exchanged freely by the natural elliptic involution, and blow up
both points equivariantly.  The involution
lifts and exchanges the two exceptional spheres.  Downstairs this operation
is one ordinary blowup at the image point.  Since the orbit was chosen away
from the fixed set, the branch surface is unchanged except for its proper
transform.  The covering manifold is
$(M_K\times S^1)\#2\CPb^2=Q_K$.
\end{proof}

This proposition identifies the ambient manifold only.  For the natural
elliptic action, the connected smoothing used to construct $\Sigma_2''$ is
sent locally to the opposite smoothing before the blowups are performed.
Consequently, no choice of a free pair of blowup points on that chosen
smoothing makes $(Q_K,\Sigma_2'')$ an equivariant marked pair.

\subsection{The ambient quotient and the marked summing surface}

The distinction can be made numerically and locally.  Let
\[
 \overline Q=(T^2\times S^2)\#\CPb^2
\]
be the quotient in Proposition~\ref{prop:QK-cover}.  The point blown up in
the quotient is chosen away from $B_K$, so the branch surface in $\overline Q$ is canonically
identified with $B_K$ and remains disjoint from the exceptional curve.  The
covering formulas give
\[
 e(Q_K)=2e(\overline Q)-e(B_K)=2,
 \qquad
 \sigma(Q_K)=2\sigma(\overline Q)-\frac12B_K^2=-2,
\]
as expected for two blowups of $M_K\times S^1$.

In the covering manifold, the two exceptional spheres form one free orbit.
This differs from
the geometry of the prescribed surface $\Sigma_2''$.  In its construction,
the two blowup centers lie on the connected smoothing of $F\cup S$, and each
exceptional sphere meets the proper transform once.  To retain a quotient of
the marked pair, the smoothing must first be invariant and the two centers
must form an invariant set.  The natural elliptic involution fails at the
first requirement.  Proposition~\ref{prop:k3-local-obstruction} below shows
that this is not a matter of moving the two centers.

There are therefore three distinct objects:
\begin{enumerate}[label=(\roman*)]
\item the branched cover $Q_K\to\overline Q$ obtained by an equivariant
blowup at a free orbit;
\item the symplectic pair $(Q_K,\Sigma_2'')$ used in the genus-two sum;
\item a hypothetical equivariant marked pair whose underlying manifold and
surface are diffeomorphic to $(Q_K,\Sigma_2'')$.
\end{enumerate}
The first exists, the second is the input to $X_3$, and the existence of the
third is the unresolved quotient problem.

Let $F$ and $S$ be a symplectic torus fiber and the zero section in
$M_K\times S^1$.  They meet transversely and positively once.  Smoothing
$F\cup S$ produces a connected symplectic surface in the class $[F]+[S]$.
It has genus two and square
\[
 ([F]+[S])^2=2F\cdot S=2.
\]
Blow up two points on this smoothed surface.  If $E_1,E_2$ are the
exceptional classes, its proper transform $\Sigma_2''\subset Q_K$ satisfies
\[
 [\Sigma_2'']=[F]+[S]-E_1-E_2,\qquad
 (\Sigma_2'')^2=2-1-1=0.
\]
Each exceptional sphere meets $\Sigma_2''$ once.  The manifold constructed
in \cite[Section~5]{AP2008} is the symplectic genus-two sum
\[
 X_3=Y_K\#_{\Sigma_2=\Sigma_2''}Q_K.
\]

\subsection{The three local quotient models and the gluing maps}

The construction of $Y_K$ uses two copies of $M_K\times S^1$, and $Q_K$
contains a third.  Before the fiber-section sums, smoothing, and blowups,
there are therefore three local quotient pairs
\[
 (T^2\times S^2,B_{\rm tr}^{(i)}),\qquad
 B_{\rm tr}^{(i)}=B_1^{(i)}\sqcup B_3^{(i)},\qquad 1\leq i\leq3.
\]
For each pair, the nontrivial base-circle monodromy is the trefoil braid of
Section~\ref{sec:piecewise-braids}; the second base-circle monodromy is
trivial.  These pre-smoothing models do not by themselves give a quotient
pair adapted to $(Q_K,\Sigma_2'')$ or a quotient of $X_3$.  The peripheral
marking and the local action at the smoothed intersection must also be
preserved.

On the $Y_K$ side the surface generators are
\[
 a^{-1}b,\qquad b^{-1}aba^{-1},\qquad d,\qquad y,
\]
and its meridian is $\mu=[x,b]$.  On the $Q_K$ side choose generators
$g,h,z$ as in \cite{AP2008}.  The boundary map used in the genus-two sum is
\[
\begin{aligned}
 \phi_{3*}(a^{-1}b)&=h,&
 \phi_{3*}(b^{-1}aba^{-1})&=z,\\
 \phi_{3*}(d)&=g^{-1}h,&
 \phi_{3*}(y)&=h^{-1}ghg^{-1},\\
 \phi_{3*}(\mu)&=(\mu'')^{-1}.
\end{aligned}
\]
This is the same boundary map as in \cite[Section~5, proof of
Lemma~13]{AP2008}, with the oriented meridian-reversing convention made
explicit.  Choices of basepoint paths may conjugate the displayed elements
and do not affect the pushout calculation.
Here $\mu''$ is a meridian of $\Sigma_2''$.  Puncturing either exceptional
sphere at its intersection with $\Sigma_2''$ gives a disk bounded by
$\mu''$.  Thus $\mu''$ is nullhomotopic in the complement and $\mu=[x,b]$
becomes trivial after the sum.

Together with the cited complement presentations, the five values of
$\phi_{3*}$ determine the pushout of the two complement groups over
$\pi_1(\Sigma_2\times S^1)$.  The relations
\[
 a^{-1}b=h,\qquad b^{-1}aba^{-1}=z,\qquad
 d=g^{-1}h,\qquad y=h^{-1}ghg^{-1}
\]
identify the two marked symplectic bases, while $[x,b]=1$ comes from the
normal-circle identification.  The last relation is geometrically supplied
by either punctured exceptional sphere.  Thus the blowups, although they
complicate equivariance, simplify the fundamental-group calculation.

\begin{lemma}\label{lem:k3-pi1}
The manifold $X_3$ is simply connected.
\end{lemma}

\begin{proof}
After translating the generator conventions, the pushout is precisely the
group computed in \cite[Section~5, Lemma~13, and Appendix]{AP2008}.  The
relations established there include
\[
\begin{aligned}
 bhb&=hb^2h,& [d,hbh^{-1}]&=1,& [h,bdb^{-1}]&=1,\\
 d^2&=b^{-1}dbdb^{-1},& b^2&=hbh^{-1}bh,&
 h^2&=dhd^{-1}hd.
\end{aligned}
\]
The first three relations give $db=bd$; the fourth then gives $b=1$, and
the last two give $h=d=1$.  The marked gluing relations give
$a=g=y=z=1$, while $dbd^{-1}=x^{-1}$ gives $x=1$.  Thus
$\pi_1(X_3)=1$.
\end{proof}

\subsection{Topology, minimality, and exoticity}

The sum formulas give
\[
 e(X_3)=6,\qquad \sigma(X_3)=-2.
\]
Together with simple connectivity this gives a unimodular form of rank four
and signature $-2$.  An even unimodular form has signature divisible by eight, so this
form is odd and hence is
$\langle1\rangle\oplus3\langle-1\rangle$.  Here $b_1=0$,
$(b_2^+,b_2^-)=(1,3)$, $X_3$ is nonspin, and the Kirby--Siebenmann
invariant vanishes because $X_3$ is smooth.  Freedman's theorem
\cite{Freedman} yields $X_3\approx\CP^2\#3\CPb^2$.  The minimality argument
in \cite[Section~5]{AP2008} applies Usher's theorem \cite{Usher}: $Y_K$ is
minimal, and each exceptional sphere in $Q_K$ meets $\Sigma_2''$ once.
Hence $X_3$ is symplectically minimal.  Since $\pi_1(X_3)=1$,
Hamilton--Kotschick \cite{HamiltonKotschick} implies that $X_3$ is
irreducible.  The standard $\CP^2\#3\CPb^2$ is reducible, so the two
manifolds are not diffeomorphic.

\subsection{The local smoothing obstruction}

Before smoothing, the zero section is fixed pointwise by the elliptic
involution.  At its intersection with the invariant fiber, choose complex
coordinates such that
\[
 S=\{w=0\},\qquad F=\{z=0\},\qquad
 \tau(z,w)=(z,-w).
\]
Thus $\tau$ preserves both axes, fixes the section pointwise, and acts by
$-1$ in the fiber direction.

\begin{proposition}\label{prop:k3-local-obstruction}
Let $B\subset\mathbb C^2$ be a sufficiently small $\tau$-invariant ball.
There is no connected $\tau$-invariant smooth surface in $B$ which is a
relative smoothing of $S\cup F$ and which agrees near its two boundary
circles with the prescribed section and fiber collars.
\end{proposition}

\begin{proof}
Suppose that such a connected surface $A$ with boundary exists.  Its collar on the section side lies
in $S=\{w=0\}$ and is fixed pointwise by $\tau$.  Average a Riemannian metric
on $A$ over the order-two action.  Then $\tau|_A$ is an isometry of the
connected surface $A$ which is the identity on a nonempty open collar.  An
isometry equal to the identity on an open set is globally the identity.
Hence
\[
 A\subset\operatorname{Fix}(\tau)=\{w=0\}.
\]
This is impossible because the other boundary collar agrees with
$F=\{z=0\}$ away from the origin, where the action is not pointwise fixed.
\end{proof}

The local annuli $\{zw=\varepsilon\}$ and $\{zw=-\varepsilon\}$ are exchanged
by $\tau$.  Their union is invariant, but has two components and is not the
connected relative smoothing used for $\Sigma_2''$.  The involution
$(z,w)\mapsto(w,z)$ preserves $\{zw=\varepsilon\}$.  Although it is conjugate
to $\tau$ after the coordinate-axis marking is forgotten, it exchanges the
two axes and does not preserve the pointwise fixed section.  Thus the
proposition is an obstruction to the natural marked elliptic action, not to
every abstract involution on a ball.

Put
\[
 \sigma(z,w)=(w,z),\qquad q_\sigma(z,w)=(s,p)=(z+w,zw).
\]
Then $\mathbb C^2/\langle\sigma\rangle\cong\mathbb C^2$, the branch curve is
$s^2=4p$, and the invariant smoothing $\{zw=\varepsilon\}$ descends to the
line $\{p=\varepsilon\}$.  The smoothing is an annulus whose quotient is a
disk, and its two fixed points map to
$s=\pm2\sqrt\varepsilon$.  Moreover, the two blowups may be made
equivariantly at a free $\sigma$-orbit on the smoothing.  The exceptional
spheres are exchanged and the quotient undergoes one ordinary blowup.
Since the global smoothing has square two, the two blowups in the covering
manifold reduce
the square of its proper transform to zero.

This swap model cannot extend over the unchanged trefoil block as an
involution interchanging the marked tori $F$ and $S$.

\begin{lemma}\label{lem:no-swap-extension}
For the right-handed trefoil, no self-diffeomorphism of $M_K\times S^1$
carries the fiber torus $F$ to the section torus $S$.
\end{lemma}

\begin{proof}
Let $A_{\rm tr}$ be the monodromy on the torus fiber.  The Wang sequence gives
\[
 H_1(M_K;\mathbb Z)\cong
 \mathbb Z\oplus\operatorname{coker}(A_{\rm tr}-I).
\]
Since $\det(A_{\rm tr}-I)=\pm1$, the cokernel vanishes and
$H_1(M_K\times S^1;\mathbb Z)\cong\mathbb Z^2$.  The map induced by the
fiber inclusion factors through this cokernel and is therefore zero.  By
contrast, the mapping-torus circle and the external product circle generate
the section torus and map isomorphically onto $H_1(M_K\times S^1)$.
Thus the fiber inclusion has rank-zero image on first homology, whereas the
section inclusion has rank-two image.  A self-diffeomorphism preserves this
rank, so it cannot carry $F$ to $S$.
\end{proof}

The two blowups occur only after the connected smoothing.  Equivariant
blowups can preserve an invariant surface, but cannot make this noninvariant
relative smoothing invariant.  Hence the third torus-bundle block has the
same local braided quotient as the first two, while the operation producing
$\Sigma_2''$ prevents the three natural local quotient descriptions from
assembling through the original construction into a global involution on
$X_3$.

\begin{theorem}\label{thm:x3-decorated}
The construction of $X_3$ admits three local branched-cover models.  Each
of its three copies of $M_K\times S^1$ carries the two-fold
quotient with branch surface $B_1\sqcup B_3$ and trefoil monodromy
$\beta_{\rm tr}$.  These are the marked pieces used in the construction of
$X_3$.  They yield a simply connected, symplectically
minimal, irreducible symplectic
manifold homeomorphic but not diffeomorphic to
$\CP^2\#3\CPb^2$.  The three natural local deck involutions do not extend
through the prescribed smoothing and gluing maps, so this statement makes no claim
that $X_3$ is a global double branched cover.
\end{theorem}

\begin{proof}
The topology follows from Lemma~\ref{lem:k3-pi1} and the calculations above.
The failure to assemble follows from
Proposition~\ref{prop:k3-local-obstruction} and
Lemma~\ref{lem:boundary-obstruction}.
\end{proof}

\section{The extension problem}
\label{sec:global-problem}

The boundary-action, relative-homology, and smoothing obstructions above
show that the natural local involutions do not assemble by the original
constructions.  They leave open constructions using a different involution
or a quotient cobordism that is not a product.

\begin{question}\label{question:global-cover}
Does any of $X_5$, $X_5^{\rm sym}$, or $X_3$ occur as a double cover of a
four-manifold diffeomorphic to a rational or ruled surface, branched over a
smoothly embedded surface?
\end{question}

\subsection{Numerical restrictions on possible quotient manifolds}

There are also numerical restrictions.  Suppose
\[
 p:X_k\longrightarrow R,\qquad k\in\{3,5\},
\]
is a double branched cover of a closed oriented four-manifold $R$, branched over a
closed embedded oriented surface $B\subset R$ with a chosen coorientation.
Then there is a class $L\in H^2(R;\mathbb Z)$ such that
$\operatorname{PD}[B]=2L$, and
\[
 e(X_k)=2e(R)-e(B),\qquad
 \sigma(X_k)=2\sigma(R)-\frac12B^2.
\]
The choice of $L$ and the branch class need not determine the cover.
If $H^1(R;\mathbb Z/2)\ne0$, the covering characters may differ by
unbranched characters.
For $X_5$ this gives
\[
 e(B)=2e(R)-8,\qquad B^2=4\sigma(R)+8,
\]
and for $X_3$ it gives
\[
 e(B)=2e(R)-6,\qquad B^2=4\sigma(R)+4.
\]
These formulas determine neither the embedding of $B$ nor the covering
character.

For the natural product candidates,
\[
 R=T^2\times S^2,\qquad T^2\times T^2,\qquad
 T^2\times\Sigma_g,
\]
one has $e(R)=\sigma(R)=0$.  Hence a branch surface for $X_5$ would have to
satisfy
\[
 e(B)=-8,\qquad B^2=8,
\]
while one for $X_3$ would have to satisfy
\[
 e(B)=-6,\qquad B^2=4.
\]
There is a further parity obstruction in the $X_3$ case.  Since
$\operatorname{PD}[B]=2L$, one has $B^2=4L^2$.  Each of
$T^2\times S^2$, $T^4$, and $T^2\times\Sigma_g$ is spin, so its
intersection form is even and $L^2$ is even.  Consequently $B^2$ is
divisible by eight, contradicting the required value $B^2=4$.  Thus none of
these unblown-up product manifolds can be the base of an oriented
double branched-cover description of $X_3$.  The value $B^2=8$ required for
$X_5$ is not excluded by this parity argument.

If the branch surface is connected and orientable, these Euler
characteristics correspond to genera five and four.  A union of torus
multisections that are unbranched over the base has Euler characteristic
zero and cannot meet these
conditions without bands or higher-genus components.  In either product-base
case the required positive self-intersection must be realized by the branch
embedding.

More generally, for
\[
 R_m=(T^2\times S^2)\#m\CPb^2
\]
the conditions become
\[
\begin{array}{c|cc}
 & e(B) & B^2\\ \hline
 X_5 & 2m-8 & 8-4m\\
 X_3 & 2m-6 & 4-4m.
\end{array}
\]
The Matsumoto base has $m=4$, giving $(e(B),B^2)=(0,-8)$ for $X_5$.
This explains the twice-tubed Matsumoto candidate of Section~\ref{sec:x5}:
the original divisor has $(4,-8)$.  Choosing oriented tubing arcs whose
interiors are disjoint from the divisor, and using the untwisted local
framing relative to normal trivializations over the attaching disks, two
internal tubes change its Euler characteristic to zero while leaving its
class and self-intersection unchanged.

The same local quotient and boundary obstruction apply to the figure-eight
block, whose braid word is $\beta_{\rm fig}=h_\alpha h_\beta^{-1}$.

\section*{Acknowledgments}

An LLM-based tool was used for limited assistance with computational
preparation, grammar checking, and figure generation.

\end{document}